 \documentclass[11pt,twoside,reqno,centertags,draft]{amsart}


 \usepackage{amsmath,amsthm,amsfonts,amssymb}
\usepackage[mathscr]{euscript}
 \usepackage{hyperref} 
 \usepackage{mathrsfs} 
\usepackage{graphicx}
\usepackage{color}
  \usepackage{epsfig}

 \newtheorem{thm}{Theorem}[section]
 
 \newtheorem{lem}[thm]{Lemma}

 \theoremstyle{definition}
 \newtheorem{defin}[thm]{Definition}
 \newtheorem{rem}[thm]{Remark}
 \newtheorem{exa}[thm]{Example}

 \newtheorem*{notations}{Notations}
 
 \DeclareMathOperator*{\esssup}{ess\,sup}

\begin{document}
 
\title{  A note on quasilinear Schr\"odinger equations 
 with singular or vanishing radial potentials }
\date{}
\maketitle     
 
\vspace{ -1\baselineskip}

{\small
\begin{center}
{\sc Marino Badiale}\footnote{Partially supported by the 
PRIN2012 grant ``Aspetti variazionali e perturbativi nei 
problemi differenziali nonlineari”.}
 and {\sc Michela Guida}\\
Dipartimento di Matematica ``Giuseppe Peano”\\
Universit\`a degli Studi di Torino\\
Via Carlo Alberto 10, 10123 Torino, Italy\\[10pt]
{\sc Sergio Rolando}\\
Dipartimento di Matematica e Applicazioni\\
Universit\`a di Milano-Bicocca\\
Via Roberto Cozzi 53, 20125 Milano, Italy  \\[10pt]
\end{center}
}

\numberwithin{equation}{section}
\allowdisplaybreaks

 \smallskip

 \begin{quote}
\footnotesize
{\bf Abstract.}  
 In this note we complete the study of \cite{BGR_quasi}, where we got existence results for the quasilinear elliptic equation
 \begin{equation*}
 	-\Delta w+ V\left( \left| x\right| \right) w - w \left( \Delta w^2 \right)= K(|x|) g(w) \quad \text{in }\mathbb{R}^{N},  
 \end{equation*}
 \noindent with singular or vanishing continuous radial potentials $V(r)$, $K(r)$. In \cite{BGR_quasi} we assumed, for technical reasons, that $K(r)$ was vanishing as $r \rightarrow 0$, while in the present paper we remove this obstruction. To face the problem we apply a suitable change of variables $w=f(u)$ and we find existence of non negative solutions by the application of variational methods. Our solutions satisfy a weak formulations of the above equation, but they are in fact classical solutions in $\mathbb{R}^{N} \setminus \{0\}$. 
 The nonlinearity $g$ has a double-power behavior, whose standard example is $g(t) = \min \{ t^{q_1 -1}, t^{q_2 -1}  \}$ ($t>0$), recovering the usual case of a single-power behavior when $q_1 = q_2$.

\end{quote}

\section{Introduction}

In this paper we complete the study of \cite{BGR_quasi}, where we got existence results for the quasilinear elliptic equation
\begin{equation}\label{EQ}
	-\Delta w+ V\left( \left| x\right| \right) w - w \left( \Delta w^2 \right)= K(|x|) g(w)  \quad \text{in }\mathbb{R}^{N}.
\end{equation}

\noindent Such an equation arises in the search of standing waves for an evolution Schr\"{o}dinger equation which has been used to study several physical phenomena (see \cite{LiuWang,PoppenbergSchmittWang, Kwon} and the references therein), such as laser beams in matter \cite{Brandi-et} and quasi-solitons in superfluids films \cite{Kuri}.
In recent times a great amount of work has been made on equation (\ref{EQ}) to overcome the non-trivial technical difficulties associated with it (see \cite{AiresSouto, ColinJeanjean, Uberlandio2, FangSzulkin, FurtadoSilvaSilva, Gloss,LiuLiuWang1, LiuLiuWang2, SilvaVieira, YangWangZhao} and the references therein). 
In \cite{BGR_quasi} and in the present paper we follow an idea introduced in \cite{LiuWang} and then used by several authors, which exploits a suitable change of variable $w=f(u)$: the problem in the new unknown $u$ can be then faced with usual variational methods, by working in an Orlicz-Sobolev space.

In almost all the papers dealing with (\ref{EQ}), the potential $V$ (be it radial or nonradial) is supposed to be positive and not vanishing at infinity. At the best of our knowledge, the only papers dealing with a potential $V$ allowed to vanish at infinity are \cite{AiresSouto,Uberlandio2,Kwon,Li-Huang}. 
In \cite{AiresSouto} and \cite{Kwon} the authors assume that $V$ is bounded, while in \cite{Li-Huang} existence of solutions is proved for possibily singular $V$'s but bounded $K$'s. In \cite{Uberlandio2}, which is the paper that inspired our work, both $V$ and $K$ can be singular or vanishing at zero or at infinity, but the authors assume that they are radial and essentially behave as powers of $|x|$ as $|x|\to 0$ and $|x|\to \infty$. 
In \cite{BGR_quasi}, instead, we have studied the case in which both $V$ and $K$ are radial potentials that can be singular or vanishing at zero or at infinity, and do not need to exhibit a powel-like behavior. The main novelties of our approach are the application of compact embeddings into the sum Lebesgue space $L^{q_1}_K + L^{q_2}_K$ (see Section 2 below) and the use of assumptions on the potentials ratio $K/V$, not on the two potentials separately.
Nevertheless, for technical reasons, in almost all cases we had to assume that $K$ was vanishing at the origin. In the present paper we remove this obstruction. 

This note is organized as follows. In Section 2 we introduce our hypotheses on $V$ and $K$, the change of variables $w=f(u)$ and the function space $E$ in which we will work. Then we state a general result concerning the compactness of the embedding of $E$ into $L_{K}^{q_{1}}+L_{K}^{q_{2}}$ (Theorem \ref{THM(cpt)}), and we give some explicit conditions ensuring that the embedding is compact (Theorems \ref{THM0} and \ref{THM1}). In Section \ref{SEC: funct} we introduce our hypotheses on the nonlinearity $g$ and we study the main properties of the functional $I$ associated to the dual problem, and in particular of its critical points, which give rise to the solutions of (\ref{EQ}). In Section \ref{SEC: ex} we apply our embedding results to get existence of nonnegative solutions to (\ref{EQ}), stating and proving the main existence result of the paper, which is Theorem \ref{THM:ex}. In section \ref{SEC:EX}, we give concrete examples of potentials $V,K$ satisfying our hypotheses and non included in the previous literature.

\begin{notations}
	
	We end this introductory section by collecting
	some notations used in the paper.

	\noindent  $\bullet $ $\mathbb{R}_{+} = ( 0, +\infty ) = \left\{ x\in \mathbb{R} : x>0 \right\}$.
	
	\noindent $\bullet $ For every $R>0$, we set $B_{R} =\left\{ x\in \mathbb{R}
	^{N}:\left| x\right| <r\right\} $.

	\noindent $\bullet $ $C_{\mathrm{c}}^{\infty }(\mathbb{R}^N )$ is the space of the
	infinitely differentiable real functions with compact support. $C_{\mathrm{c}, r}^{\infty }( \mathbb{R}^N)$ is the subspace 
	of $C_{\mathrm{c}}^{\infty }(\mathbb{R}^N)$ made of radial functions.
	
	\noindent $\bullet $ For any measurable set $A\subseteq \mathbb{R}^{N}$, $
	L^{q}(A)$ and $L_{\mathrm{loc}}^{q}(A)$ are the usual real Lebesgue spaces.
	If $\rho :A\rightarrow \Bbb{R}_{+}$ is a measurable function, then $%
	L^{p}(A,\rho \left( z\right) dz)$ is the real Lebesgue space with respect to
	the measure $\rho \left( z\right) dz$ ($dz$ stands for the Lebesgue measure
	on $\mathbb{R}^{N}$). In particular, if $K:\Bbb{R}_{+}\rightarrow \Bbb{R}
	_{+} $ is measurable, we denote $L_{K}^{q}\left( A\right) :=L^{q}\left(
	A,K\left( \left| x\right| \right) dx\right) $.
	
	\noindent $\bullet $ If $Y$ is a Banach space, $Y'$ is its dual.

\end{notations}

\section{Hypotheses and preliminary results  }

\par \noindent Throughout this paper we will assume $N\geq 3$ and the following hypotheses on $V,K$:

\begin{itemize}
	
	\item[$\left( \mathbf{H}\right) $]  $V:\mathbb{R}_{+}\rightarrow
	\left[ 0,+\infty \right) $ and $K:\mathbb{R}_{+} \rightarrow \mathbb{R}_{+} $ are continuous, and there exists $C>0$ such that for all $r \in (0,1)$ one has
	$$
	V(r) \leq \frac{C}{r^2} .
	$$

\end{itemize}

As a first thing we introduce the function that we need to define the Orlicz-Sobolev space in which we will work. Such a function is defined as the solution $f$ of the following Cauchy problem:
\begin{equation}
	\label{eq:change}
	\left\{ 
	\begin{array}{ll}
		f'(t)= \frac{1}{\sqrt{1+2f^2 (t)}} \quad \quad  \text{in }\mathbb{R} \\ 
		
		f(0)=0    
	\end{array}
	\right. 
\end{equation}
It is easy to check that this problem has a unique solution $f\in C^{\infty}(\mathbb{R}, \mathbb{R})$, which is odd, strictly increasing, and surjective (whence invertible). Other important properties of $f$ are listed in Lemma 2.1 of \cite{BGR_quasi}.
We use the function $f$ to define a suitable change of unknown, which is the following: we call $w$ the solution of (\ref{EQ}) we are looking for and we set $w= f(u)$, where $u$ is the new unknown, living in a suitable space that we are going to define. To get solutions $w$ to (\ref{EQ}) we will look for solutions $u$ to the following equation:
\begin{equation}\label{EQdual}
	-\Delta u+ V\left( \left| x\right| \right) f(u) f' (u) = K(|x|) g(f(u)) f' (u) \quad \text{in }\mathbb{R}^{N},  
\end{equation}
which will be obtained as critical points of the following functional:
\begin{equation}\label{funct1}
	I(u) = \frac{1}{2}\int_{\mathbb{R}^{N}}|\nabla u |^2 dx + \frac{1}{2}\int_{\mathbb{R}^{N}}V(|x|) f^2 (u)  dx-\int_{\mathbb{R}^{N}}K(|x|) G(f (u)) \,dx. 
\end{equation}
The critical points of $I$ and their relations with solutions of (\ref{EQ}) will be studied in the next section. In the rest of the present section we introduce the function space $E$ in which we will obtain the critical points of $I$, and we study the relevant  compactness results for $E$. All the results of this section have been proven in \cite{Uberlandio2} and \cite{BGR_quasi}.

To this aim, we introduce the space $D_r^{1,2} \left(\mathbb{R}^{N}  \right) $ as the closure of $C_{\mathrm{c}, r}^{\infty }( \mathbb{R}^{N} ) $ with respect to the norm $||u||_{1,2} := \left( \int_{\mathbb{R}^{N}} |\nabla u|^2 
dx \right)^{1/2}$. It is well known that $D_r^{1,2} \left(\mathbb{R}^{N}  \right) $ is a Hilbert space. Then we define the space $E$ as follows:
$$E = \left\{ u \in  D_r^{1,2} \left(\mathbb{R}^{N}\right) \, \Big| \, \int_{\mathbb{R}^{N}} V(|x|) f^2 (u) dx <+\infty \right\}.$$
In $E$ we introduce the Orlicz norm
$$
||u||_o = \inf_{k>0} \frac{1}{k} \left[ 1+  \int_{\mathbb{R}^{N}} V(|x|) f^2 ( k u) dx    \right] ,
$$
and then we set
$$||u || = ||u||_{1,2} + ||u||_o .$$
The space $E$ endowed with the norm $|| . ||$ is an Orlicz-Sobolev space. In the following theorems we recall some of its properties. For the proofs, and other relevant properties of $E$, see \cite{BGR_quasi}.

\begin{thm}
	$\left( E, ||.||\right)$ is a Banach space and the following continuous embedding holds true:
	$$
	E\hookrightarrow D_r^{1,2} \left(\mathbb{R}^{N}\right) .
	$$
\end{thm}

\begin{lem}
	\label{lem:properties}
	If $u_n (x) \rightarrow u(x)$ a.e. in $\mathbb{R}^{N}$ and 
	$$
	\int_{\mathbb{R}^{N}} V(|x|) f^2 (u_n ) \, dx \rightarrow \int_{\mathbb{R}^{N}} V(|x|)  f^2 (u ) \, dx ,
	$$
	then $||u_n - u  ||_o  \rightarrow 0$.
\end{lem}

We now state the main compactness results concerning the space $E$, which have been proven in \cite{BGR_quasi}. They concern the embedding properties of $E$ into the sum space 
\[
L_{K}^{q_{1}}+L_{K}^{q_{2}}:=\left\{ u_{1}+u_{2}:u_{1}\in
L_{K}^{q_{1}}\left( \Bbb{R}^{N}\right) ,\,u_{2}\in L_{K}^{q_{2}}\left( \Bbb{R%
}^{N}\right) \right\} ,\quad 1<q_{i}<\infty . 
\]
We recall from \cite{BPR} that such a space can be characterized as the set
of measurable mappings $u:\Bbb{R}^{N}\rightarrow \Bbb{R}$ for which there
exists a measurable set $A\subseteq \Bbb{R}^{N}$ such that $u\in
L_{K}^{q_{1}}\left( A\right) \cap L_{K}^{q_{2}}\left( A^{c}\right) $. It is
a Banach space with respect to the norm 
\[
\left\| u\right\| _{L_{K}^{q_{1}}+L_{K}^{q_{2}}}:=\inf_{u_{1}+u_{2}=u}\max
\left\{ \left\| u_{1}\right\| _{L_{K}^{q_{1}}(\Bbb{R}^{N})},\left\|
u_{2}\right\| _{L_{K}^{q_{2}}(\Bbb{R}^{N})}\right\} 
\]
and the continuous embedding $L_{K}^{q}\hookrightarrow
L_{K}^{q_{1}}+L_{K}^{q_{2}}$ holds true for all $q\in \left[ \min \left\{
q_{1},q_{2}\right\}, \max \left\{ q_{1},q_{2}\right\} \right] $. 
The first compactness result is Theorem \ref{THM(cpt)} below. Its assumptions are rather general but not so easy to check, so more handy conditions ensuring such assumptions will be provided in the next Theorems \ref{THM0} and \ref{THM1}. To state the results, we need to preliminarly introduce the following functions of $R>0$ and $q>1$:
\begin{eqnarray}
	\mathcal{S}_{0}\left( q,R\right)&:=&
	\sup_
	{u\in E,\,
		\left\| u\right\| =1  }
	\int_{B_{R}}K\left( \left| x\right| \right)
	\left| u\right| ^{q}dx,  \label{S_o :=}
	\\
	\mathcal{S}_{\infty }\left( q,R\right)&:=&
	\sup_
	{u\in E,\,
		\left\| u\right\| =1  } 
	\int_{\mathbb{R}%
		^{N}\setminus B_{R}}K\left( \left| x\right| \right) \left| u\right| ^{q}dx.
	\label{S_i :=}
\end{eqnarray}
Clearly $\mathcal{S}_{0}\left( q,\cdot \right) $ is nondecreasing, $\mathcal{
	S}_{\infty }\left( q,\cdot \right) $ is nonincreasing and both of them can
be infinite at some $R$.

\begin{thm}
	\label{THM(cpt)} Let $N\geq 3$, let $V$ and $K$ be as in $\left( \mathbf{H}\right) $ and let $q_{1},q_{2}>1$.
	If 
	\begin{equation}\label{eqcompact}
		\lim_{R\rightarrow 0^{+}}\mathcal{S}_{0}\left( q_{1},R\right)
		=\lim_{R\rightarrow +\infty }\mathcal{S}_{\infty }\left( q_{2},R\right) =0, 
	\end{equation}
	then $E$ is compactly embedded into $L_{K}^{q_{1}}(\mathbb{R}^{N})+L_{K}^{q_{2}}(\mathbb{R}^{N})$.
	
\end{thm}

\noindent We notice that assumption \eqref{eqcompact}
can hold with $q_{1}=q_{2}=q$ and therefore Theorem \ref{THM(cpt)} also
concerns the compact embedding properties of $E$ into $L_{K}^{q}$, $1<q<\infty $.

We now look for explicit conditions on $V$ and $K$ implying \eqref{eqcompact} for some $q_{1}$ and $q_{2}$. More
precisely, in Theorem \ref{THM0} we will find a range of exponents $
q_{1} $ such that $\lim_{R\rightarrow 0^{+}}\mathcal{S}_{0}\left(q_{1},R\right)$ $=0$, while 
in Theorem \ref{THM1} we will do the same for exponents $q_{2}$ such that
$\lim_{R\rightarrow +\infty}\mathcal{S}_{\infty }\left( q_{2},R\right) =0$.

For $\alpha \in \mathbb{R}$, $\beta \in \left[ 0,1\right] $, we define two functions 
$q_0^{*} \left( \alpha ,\beta\right) $ and $q_{\infty}^{*}\left( \alpha ,\beta\right)$ by setting 
$$\quad q_0^{*}\left( \alpha ,\beta \right) :=\frac{2 \alpha  +2 N - \beta (N+2)}{N-2}, 
\quad \quad q_{\infty}^{*}\left( \alpha ,\beta \right) :=2 \, \frac{ \alpha  + N - 2\beta }{N-2}.
$$

\begin{thm}
	\label{THM0}
	Let $V$, $K$ be as in $\left( \mathbf{H}\right) $.
	Assume that there exists $R_{1}>0$ such that
	\begin{equation} \label{esssup in 0}
		\sup_{r\in \left( 0,R_{1}\right) }\frac{K\left( r\right) }{%
			r^{\alpha _{0}}V\left( r\right) ^{\beta _{0}}}<+\infty \quad \text{for some }%
		0\leq \beta _{0}\leq 1\text{~and }\alpha _{0} \in \mathbb{R} . 
	\end{equation}
	Assume also that 
	$$
	\max \left\{ 1,2\beta _{0}\right\} <q_0^{*}  \left( \alpha _{0},\beta
	_{0}\right). 
	$$
	Then
	$\displaystyle \lim_{R\rightarrow 0^{+}}\mathcal{S}_{0}\left(
	q_{1},R\right) =0$ for every $q_{1}\in \mathbb{R}$ such that 
	\begin{equation}
		\max \left\{ 1,2\beta _{0}\right\} <q_{1}<q_0^{*}  \left( \alpha _{0},\beta
		_{0}\right) .  \label{th1}
	\end{equation}
\end{thm}

\begin{thm}
	\label{THM1}
	Let $V$, $K$ be as in $\left( \mathbf{H}\right) $.
	Assume that there exists $R_{2}>0$ such that
	\begin{equation}
		\sup_{r>R_{2}}\frac{K\left( r\right) }{r^{\alpha _{\infty
			}}V\left( r\right) ^{\beta _{\infty }}}<+\infty \quad \text{for some }0\leq
		\beta _{\infty }\leq 1\text{~and }\alpha _{\infty }\in \mathbb{R}.
		\label{esssup all'inf}
	\end{equation}
	Then $\displaystyle \lim_{R\rightarrow +\infty }\mathcal{S}_{\infty }\left(
	q_{2},R\right) =0$ for every $q_{2}\in \mathbb{R}$ such that 
	\begin{equation}
		q_{2}>\max \left\{ 1,2\beta _{\infty },q_{\infty}^{*}\left(\alpha _{\infty },\beta
		_{\infty }\right) \right\} .  \label{th2}
	\end{equation}
\end{thm}

\begin{rem}
	We mean $V\left( r\right) ^{0}=1$ for every $r$
	(even if $V\left( r\right) =0$). In particular, if $V\left( r\right) =0$ for $r>R_{2}$, then Theorem \ref{THM1} can be applied with $\beta
	_{\infty }=0$ and assumption (\ref{esssup all'inf}) means 
	\[
	\esssup_{r>R_{2}}\frac{K\left( r\right) }{r^{\alpha _{\infty }}}%
	<+\infty \quad \text{for some }\alpha _{\infty }\in \mathbb{R}.
	\]
	Similarly for Theorem \ref{THM0} and assumption (\ref{esssup in 0}), if $%
	V\left( r\right) =0$ for $r\in \left( 0,R_{1}\right) $.
\end{rem}


\section{The functional I \label{SEC: funct}}

In this section we study the functional $I$, defined in (\ref{funct1}), whose critical points will give rise to solutions to \eqref{EQdual}, and thus to \eqref{EQ}. 
To this aim, we need a set of hypotheses on the nonlinearity $g$. 

First of all we make the following assumptions, to be maintained from now on:

\begin{itemize}
	\item[$ \bf (h_1 )$] we assume both hypotheses  (\ref{esssup in 0}) and (\ref{esssup all'inf}) of Theorems \ref{THM0} and \ref{THM1};

	\item[ $\bf (h_2 )$] we let $q_1 , q_2 \in \mathbb{R}$ be such that $4 < q_1 < 2 q_0^* (\alpha_0 , \beta_0 ) $ and $q_2 > \max \left\{ 4, 2 q_{\infty}^* (\alpha_{\infty} , \beta_{\infty} ) \right\}$.

\end{itemize}

\noindent Then we introduce the conditions we will use about the continuous $g$, meaning $G(t)= \int_0^t g(s) \, ds$:

\begin{itemize}

	\item[${\bf \left( g_{1}\right) } $] $g : \mathbb{R} \rightarrow \mathbb{R}$ is a continuous function;
	
	\item[${\bf \left( g_{2}\right) } $]  $\exists \theta >2$ such that $0\leq 2\theta
	G\left( t\right) \leq g\left( t\right) t$ for all $t\in \mathbb{R}$;
	
	\item[${\bf  \left( g_{3}\right)} $]  $\exists t_{0}>0$ such that $G\left(t_{0}\right) >0$;
	
	\item[${\bf  \left( g_{q_{1},q_{2}}\right) }$]  $\exists C>0$ such that
	$\left|g\left( t\right) \right| \leq C \min \left\{ \left| t\right|
	^{q_{1}-1},\left| t\right| ^{q_{2}-1}\right\} $ for all $t\in \mathbb{R}$.

\end{itemize}

\noindent We notice that these hypotheses imply that $q_1, q_2 \geq 2\theta $, and that there exists $C>0$ such that the following estimate holds for all $t\in \mathbb{R}$: 
\begin{equation}\label{estimG}
	\left|G\left( t\right) \right| \leq C \min \left\{ \left| t\right|
	^{q_{1}},\left| t\right| ^{q_{2}}\right\}.
\end{equation}
We also observe that, if $q_{1}\neq q_{2}$, the double-power growth
condition ${\bf \left( g_{q_{1},q_{2}}\right)} $ is more stringent than the more
usual single-power one, since it implies $|g(t)| \leq C |t|^{q-1}$
for $q=q_{1}$, $q=q_{2}$ and every $q$
in between. On the other hand, we will never require $q_{1}\neq q_{2}$ in ${\bf \left( g_{q_{1},q_{2}}\right)} $, so that our results will also concern
single-power nonlinearities as long as we can take $q_{1}=q_{2}$.
%

We begin with the following lemma.

\begin{lem}\label{embedmezzi}
	Assume $ \bf (h_1 )$, $ \bf (h_2 )$. Then $E $ is compaclty embedded into $L_K^{q_1/2} \left( \mathbb{R}^N \right) + L_K^{q_2/2} \left( \mathbb{R}^N \right)$.
\end{lem}

\begin{proof}
	The hypothesis $\bf \left( h_2 \right)$ easily gives that $q_1 /2$ and $q_2 /2$ satisfy  (\ref{th1}) and  (\ref{th2}). Together with $\bf \left( h_1 \right)$, this implies that the hypotheses of Theorems \ref{THM0} and \ref{THM1} are satisfied. Together with Theorem \ref{THM(cpt)}, this gives the result.
\end{proof}


We now state the main properties of the functional $I:E\rightarrow \mathbb{R}$ defined by (\ref{funct1}).

\begin{thm} 
	\label{THM:critical}
	Let $N\geq 3$ and assume $\left( \mathbf{H}\right) $, $(\bf h_1 )$, $(\bf h_2 )$. Assume that $g: {\mathbb{R}} \rightarrow {\mathbb{R}}$ 
	satisfies ${\bf \left( g_{1}\right) } $, ${\bf \left( g_{2}\right) } $, ${\bf \left( g_{3}\right) } $ and ${\bf  \left( g_{q_{1},q_{2}}\right) }$. Then 
	\begin{itemize}
		\item[(1)] $I$ is well defined and continuous in $E$;
		
		\item[(2)] $I$ is a $C^1$ map on $E$ and for all $u \in E$ its differential $I' (u)$ is given by
		\begin{align} 
			I' (u) h  = & \int_{\mathbb{R}^{N}}\nabla u \nabla h dx + \int_{\mathbb{R}^{N}}V(|x|) f (u) f'(u) h  dx \label{diffrechet}\\
			&-\int_{\mathbb{R}^{N}}K(|x|) g(f (u)) f'(u) h\,dx \nonumber
		\end{align}
		for all $h \in E$.
		
	\end{itemize}
	
\end{thm}

\begin{proof}
	Define
	$$ I_1 (u)  =    \frac{1}{2} \int_{\mathbb{R}^{N}} |\nabla u |^2 dx, \qquad I_2 (u)  = \frac{1}{2} \int_{\mathbb{R}^{N}}V(|x|) f^2 (u)  dx , $$
	$$I_3 (u) = 
	\int_{\mathbb{R}^{N}}K(|x|) G(f (u))  \,dx ,$$
	where $G(t)= \int_{0}^{t} g(s)ds $. We study these three functionals separately.
	\par \noindent As to $I_1$, it is a standard task to get that $I_1$ is $C^1$ on $E$ with differential given by $I'_1 (u)h= \int_{\mathbb{R}^{N}}\nabla u \nabla h dx $. 
	\par \noindent As to $I_3$, we notice that, setting $h(x,t)= K(|x|) G(f(t))$, we have $h(x,t) = \int_{0}^{t} K(|x|) g(f (s)) f'(s) ds$. Exploiting the properties of $f$ (see Lemma 2.1 in \cite{BGR_quasi}), we have that
	$$
	|g(f(t))| \, |f'(t) | \leq C |f(t)|^{q_i -1} |t|^{-1} |t f'(t)| \leq C |f(t)|^{q_i } |t|^{-1}\leq C |t|^{q_i/2 -1 } .
	$$
	and therefore
	$$\left|  K(|x|) g(f (t)) f'(t) 	\right| \leq C  K(|x|)\min \left\{ \left| t\right|
	^{q_{1}/2-1},\left| t\right| ^{q_{2}/2-1}\right\} .$$
	Then we can apply the results of \cite{BPR} (in particular Proposition 3.8) and the fact that $E \hookrightarrow L_{K}^{q_{1}}(\mathbb{R}^{N})+L_{K}^{q_{2}}(\mathbb{R}^{N})$  (see Theorem \ref{THM(cpt)}), to get that also $I_3$ is $C^1$ on $E$, with differential given by 
	$$I'_3 (u)h = \int_{\mathbb{R}^{N}}K(|x|) g(f (u)) f'(u) h\,dx.$$
	\par \noindent As to $I_2$, we can repeat here the arguments of \cite{Uberlandio2} or \cite{BGR_quasi}, which also work under our hypotheses, to get that $I_2$ is well defined, continuous and Gateaux differentiable on $E$, with differential $I'_2$ given by
	$$  I'_2 (u) h = \int_{\mathbb{R}^{N}}V(|x|) f (u) f'(u) h  dx .$$
\end{proof}

We now give the main properties of the critical points of $I$, under the hypotheses of Theorem \ref{THM:critical}. In particular, we will point out the relation between equation (\ref{EQdual}) and the original equation (\ref{EQ}).  
Clearly a critical point $u$ of $I$ satisfies $I'(u)h=0$, i.e.
\begin{equation}
	\label{EQweakdual}
	\int_{\mathbb{R}^{N}}\nabla u \nabla h \, dx + \int_{\mathbb{R}^{N}}V(|x|) f (u) f'(u) h \,  dx -\int_{\mathbb{R}^{N}}K(|x|) g(f (u)) f'(u) h\,dx =0
\end{equation}
for all $h \in E$, which is, of course, a weak formulation of equation (\ref{EQdual}). The other relevant properties of the critical points of $I$ are summarized by the following theorem, which we proved in \cite{BGR_quasi}. 

\begin{thm} 
	Assume the hypotheses of Theorem \ref{THM:critical}. Let $u\in E$ be a critical point of $I$ and set $w= f(u)$. Then
	\begin{itemize}
		\item[(1)] $u \in  C^2 ( \mathbb{R}^{N} \backslash \{ 0 \} )$ and $u$ is a classical solution of equation (\ref{EQdual}) in $ \mathbb{R}^{N} \backslash \{ 0 \}$;
		
		\item[(2)] $w \in C^2 \left(\mathbb{R}^{N} \backslash \{ 0\} \right) $ and $w$ is a classical solution of equation (\ref{EQ}) in $\mathbb{R}^{N} \backslash \{ 0\}$;
		
		\item[(3)] $w \in X$ and $w$ satisfies the following weak formulation of (\ref{EQ}):
		\begin{equation*} 
			\int_{\mathbb{R}^{N}}  \left( 1+ 2w^2 \right) \nabla w\cdot \nabla h\, dx+ \int_{\mathbb{R}^{N}}   2w |\nabla w|^2 \, h \, dx + \int_{
				\mathbb{R}^{N}}V\left( \left| x\right| \right) wh\,dx=
		\end{equation*}
		\begin{equation} \label{EQweak2}
		=\int_{		\mathbb{R}^{N}}K(|x|) g(w)  h\,dx 
		\end{equation}
		for all $h \in C_{\mathrm{c}, r}^{\infty }( \mathbb{R}^{N} )$.
		
	\end{itemize}
	
\end{thm}

%
%
%
%
%
%
%
%
%
%
%
%
%
%
%
%
%
%
%
%

%

\section{Existence of solutions \label{SEC: ex}}

Our main existence result is the following.

\begin{thm}
	\label{THM:ex} Assume the hypotheses of Theorem \ref{THM:critical}. Then the functional $I:E\rightarrow \mathbb{R}$ has a nonnegative
	critical point $u\neq 0$.
\end{thm}

\begin{rem}
	\label{RMK:thm:ex} In Theorem \ref{THM:ex}, as we look for nonnegative solutions, we may assume $g(t)=0$ for all $t\leq 0$. Indeed, if we have a nonlinearity $g$ satisfying the hypotheses of the theorem, we can replace $g(t)$ with $\chi _{\mathbb{R}_{+}}(t) g(t) $ ($\chi _{\mathbb{R}_{+}}$ is the characteristic function of $\mathbb{R}_{+}$) and the new nonlinearity still satisfies the hypotheses.
\end{rem}


\begin{rem}
	As concerns examples of nonlinearities satisfying the hypotheses of Theorem \ref{THM:ex}, the simplest 
	$g\in C\left(\mathbb{R};\mathbb{R}\right) $ such that ${\bf  \left( g_{q_{1},q_{2}}\right) }$ holds is 
	\[g\left( t\right) =\min \left\{ \left| t\right| ^{q_{1}-2}t,\left| t\right|
	^{q_{2}-2}t\right\} ,
	\]
	which also ensures ${\bf  \left( g_{2}\right) }$ if $
	q_{1},q_{2}>4$ (with $\theta = \min \left\{ \frac{q_1}{2}, \frac{q_2}{2} \right\} $). Another model example is 
	\[
	g\left( t\right) =\frac{\left| t\right| ^{q_{2}-2}t}{1+\left| t\right|
		^{q_{2}-q_{1}}}\quad \text{with }1<q_{1}\leq q_{2},
	\]
	which ensures ${\bf \left( g_{2}\right) }$ if $q_{1}>4$ (with $\theta =\frac{q_1}{2}$). Note that, in
	both these cases, also ${\bf \left( g_{3}\right) }$ holds true. Moreover, both of these functions $g$
	become $g\left( t\right) =\left| t\right| ^{q-2}t$ if $q_{1}=q_{2}=q$. 
\end{rem}

In order to prove Theorem \ref{THM:ex}, we first recall the Palais-Smale condition and a version of the well-known Mountain-Pass Lemma (see \cite[Chapter 2]{AubinEkeland}). 

\begin{defin}
	Let $Y$ be a Banach space and $\Phi : Y \rightarrow \mathbb{R}$ a $C^1$ functional. We say that $\Phi $ satisfies the Palais-Smale condition if for any sequence $\{ x_n \}_n $ sucht that $\Phi(x_n )$ is bounded in $\mathbb{R}$ and $\Phi' (x_n ) \rightarrow 0$ in $Y'$, there exists a subsequence $\{ x_{n_k} \}_k $ converging in $Y$.
\end{defin}

\begin{thm} (Mountain Pass Lemma) \label{MPLemma}
	
	\noindent Let $Y$ be a Banach space and $\Phi : Y \rightarrow \mathbb{R}$ a $C^1$ functional such that $\Phi (0)=0$ . Assume that $\Phi$ satisfies the Palais-Smale condition and that there exist a subset $S \subseteq Y$ and $\alpha >0$  such that: 
	\begin{itemize}
		
		\item[(i)] $Y  \setminus S$ is not arcwise connected;
		
		\item[(ii)] $\Phi  (x) \geq \alpha$ for all $x \in S$;
		
		\item[(iii)] there exists $y \in Y  \backslash (C_0 \cup S )$ such that $\Phi (y)<0$, where $C_0$ is the connected component of $ Y  \backslash S$ such that $0 \in C_0$.
		
	\end{itemize}
	Then $\Phi $ has a critical point $u \in Y$ such that $\Phi (u) \geq \alpha $.

\end{thm}

We will prove Theorem \ref{THM:ex} by showing that the functional $I: E\rightarrow \mathbb{R}$ satisfies the hypotheses of the Mountain Pass Lemma.  Obviously $I(0)=0$. The proof of the other hypotheses is the subject of Lemmas \ref{LE:MP0} and \ref{LE:MP4}. 
Recall the three functionals $I_1 , I_2 , I_3$ introduced in the proof of Lemma \ref{THM:critical}, and for $u \in E$ define $J (u) = I_1 (u)+ I_2 (u)= \frac{1}{2} \int_{\mathbb{R}^{N}}|\nabla u |^2 \, dx  + \frac{1}{2} \int_{\mathbb{R}^{N}}V(|x|) f^2 (u) \,dx $. Moreover, for any $\rho >0$ set 
$$
S_{\rho}= \left\{ u \in E \, | \, J(u) = \rho \right\}.
$$


\begin{lem}
	\label{LE:MP0}
	Assume the hypotheses of Theorem \ref{THM:ex}. Then:
	\begin{itemize}
		
		\item[(1)] for every $\rho >0$ the set $E\setminus S_{\rho}$ is not arcwise connected;
		
		\item[(2)] there exist $\rho , \alpha >0$ and $v\in E$ such that $J(v) > \rho$, $I(v)<0$ and $I(u) \geq \alpha$ for all $u \in S_{\rho}$;
		
		\item[(3)] if $\left\{ u_n   \right\}_n \subseteq E$ is a Palais-Smale sequence for $I$, i.e., a sequence such that 
		$ \left\{ I(u_n ) \right\}_n$ is bounded and $ I' (u_n ) \rightarrow 0 $ in $E'$, then $\left\{ u_n   \right\}_n $ in bounded in $ E$.
		
	\end{itemize}
	
\end{lem}

\proof
See \cite{BGR_quasi}.
\endproof

%
%
%
%
%
%
%
%
%
%
%
%
%
%
%
%
%


The following lemma shows that $I$ satisfies the Palais-Smale condition. This relies on the compact embedding of $E$ into $L_{K}^{q_{1}/2}(\mathbb{R}^{N})+L_{K}^{q_{2}/2}(\mathbb{R}^{N})$, which is one of the main devices in our proof of Theorem \ref{THM:ex} and holds true by Lemma \ref{embedmezzi}.

\begin{lem}
	\label{LE:MP4}
	\label{LEM:PS}Under the assumptions of Theorem \ref{THM:ex}, the functional $
	I:E\rightarrow \mathbb{R}$ satisfies the Palais-Smale condition.
\end{lem}

\proof

Let $\left\{ u_{n}\right\}_n $ be a sequence in $E$ such that $\left\{
I\left( u_{n}\right) \right\}_n $ is bounded and $I^{\prime }\left(
u_{n}\right) \rightarrow 0$ in $E^{\prime }$. From the previous lemma we get that $\left\{ u_{n}\right\}_n $ is bounded in $E$, so that there is a subsequence, which we still call $\left\{ u_{n}\right\}_n $, such that $u_n \rightharpoonup u$ in $E$ and in $D_r^{1,2} \left(\mathbb{R}^{N}\right) $, and $u_n (x) \rightarrow u(x)$ a.e.$x$. Recall 
that we have defined $J= I_1 + I_2$, so that 
$I= J-I_3$. We know that $I_3$ is $C^1$ as a functional from $L_{K}^{q_{1}/2}+L_{K}^{q_{2}/2}$ to $\mathbb{R}$. By compactness of the embedding of $E$ into $L_{K}^{q_{1}/2}+L_{K}^{q_{2}/2}$, we have that $u_n \rightarrow u$ in $L_{K}^{q_{1}/2}+L_{K}^{q_{2}/2}$. Hence $I_{3}' (u_n ) \rightarrow I_{3}'(u)$ in the dual space of $L_{K}^{q_{1}/2}+L_{K}^{q_{2}/2}$ and $I_{3}' (u_n )(u-u_n )\rightarrow 0$ in $\mathbb{R}$.
We now notice that, as $f^2$ is a convex function (see Lemma 2.1 in \cite{BGR_quasi}), $J$ is a convex functional on $E$. Hence
$$
J(u)-J(u_n ) \geq J'( u_n ) (u- u_n ) = I'( u_n ) (u- u_n ) + I_{3}'( u_n ) (u- u_n ).
$$
Since $I'( u_n ) \rightarrow 0$ in $E'$ and $\left\{ u-u_{n}\right\} $ is bounded in $E$, we have that $I'( u_n ) (u- u_n )\rightarrow 0$. So 
$$
J(u) \geq J(u_n ) + o(1).
$$
Taking the $\liminf_n$, this gives
\begin{equation}
	\label{EQPS1}
	\int_{\mathbb{R}^{N}}| \nabla u |^2 dx + \int_{\mathbb{R}^{N}}V(|x|) f^2 (u ) dx 
\end{equation}
$$
\geq \liminf_n \left( \int_{\mathbb{R}^{N}}| \nabla u_n |^2 dx + \int_{\mathbb{R}^{N}}V(|x|) f^2 (u_n ) dx \right) 
$$
$$\geq \liminf_n  \int_{\mathbb{R}^{N}}| \nabla u_n |^2 dx + \liminf_n \int_{\mathbb{R}^{N}}V(|x|) f^2 (u_n ) dx .
$$
By semicontinuity of the norm, one has
$$
\liminf_n \int_{\mathbb{R}^{N}}| \nabla u_n |^2 dx \geq \int_{\mathbb{R}^{N}}| \nabla u |^2 dx,
$$
and hence (\ref{EQPS1}) gives
$$
\int_{\mathbb{R}^{N}}V(|x|) f^2 (u ) dx \geq \liminf_n  \int_{\mathbb{R}^{N}}V(|x|) f^2 (u_n ) dx.
$$
Fatou's Lemma obviously implies 
$$\int_{\mathbb{R}^{N}}V(|x|) f^2 (u ) dx \leq \liminf_n  \int_{\mathbb{R}^{N}}V(|x|) f^2 (u_n ) dx,$$
whence
\begin{equation}
	\label{EQPS2}
	\int_{\mathbb{R}^{N}}V(|x|) f^2 (u ) dx = \liminf_n  \int_{\mathbb{R}^{N}}V(|x|) f^2 (u_n ) dx.
\end{equation}
Passing to a subsequence, which we still label $\left\{ u_{n}\right\}_n $, we can assume
\begin{equation}
	\label{EQPS3}
	\int_{\mathbb{R}^{N}}V(|x|) f^2 (u ) dx = \lim_n \int_{\mathbb{R}^{N}}V(|x|) f^2 (u_n ) dx.
\end{equation}
By Lemma \ref{lem:properties}, we infer that $||u- u_n ||_o \rightarrow 0$. Repeating the previous arguments for this subsequence, we get again (\ref{EQPS1}). which now gives
$$
\int_{\mathbb{R}^{N}}| \nabla u |^2 dx \geq \liminf_n \int_{\mathbb{R}^{N}}| \nabla u_n |^2 dx 
$$
and hence 
$$
\int_{\mathbb{R}^{N}}| \nabla u |^2 dx = \liminf_n \int_{\mathbb{R}^{N}}| \nabla u_n |^2 dx . 
$$
Passing again to a subsequence if necessary, we may assume
$$
\int_{\mathbb{R}^{N}}| \nabla u |^2 dx = \lim_n \int_{\mathbb{R}^{N}}| \nabla u_n |^2 dx . 
$$
\noindent As $u_n \rightharpoonup u$ in $D_r^{1,2} \left(\mathbb{R}^{N}\right) $, we conclude that $u_n \rightarrow u$ in $D_r^{1,2} \left(\mathbb{R}^{N}\right) $, that is 
$||u- u_n ||_{1,2} \rightarrow 0$, and therefore $||u - u_n || = ||u- u_n ||_o + ||u- u_n ||_{1,2}\rightarrow 0$. 
\endproof

We can now easily conclude the proof of Theorem \ref{THM:ex}.

\proof[Proof of Theorem \ref{THM:ex}.] 
Taking $\rho$ and $v$ as in $(2)$ of Lemma \ref{LE:MP0}, we have $0=J(0) <\rho<J(v)$, so that  $v$ and $0$ are in two distinct 
connected components of $E1setminus S_{\rho}$. Together with Lemma \ref{LE:MP0}, this shows that all the hypoteses of the Mountain Pass Lemma \ref{MPLemma} are satisfied, so that we get a critical point $u \not= 0$ of $I$. As $g(t)=0$ for $t<0$ (cf. Remark \ref{RMK:thm:ex}), it is a standard task to get that the negative part $u^{-}$ of $u$ is zero (see for example the proof of Theorem 7.1 in \cite{BGR_quasi}), that is, $u$ is nonnegative.
\endproof

\section{Examples}\label{SEC:EX}

\noindent In this section we give some examples of applications of our existence results. In all the analogous applications (except one) contained in our previous paper \cite{BGR_quasi}, we had to assume that $K$ vanishes as $r \rightarrow 0$. Therefore, here we limit ourselves to give examples in which this does not happen. We also point out that all our examples do not satisfy the hypotheses of the results of \cite{Uberlandio2}, which is the other main reference about the problem under consideration. 

As usual, we let $N\geq 3$ and assume hypothesis ${\bf (H)}$. Just for simplicity, in all the examples we shall consider the model nonlinearity $g(t)= \min \{ t^{q_1 -1}, t^{q_2 -1} \}$. 

\begin{exa} \label{ex1}
	This first example is a very simple one: set $K(r)= 1 $ and $V(r)= 1/r^2$ for all $r>0$. Computing the coefficients of \cite{Uberlandio2}, we find $a_0 = -2$ and $b_0 = 0$, so that the results in \cite{Uberlandio2} cannot be applied, because they need $a_0 \geq b_0$ (see \cite[Theorem 1.1]{Uberlandio2}). Instead, if we set $\alpha_0 = \alpha_{\infty}= \beta_0 = \beta_{\infty}=0$, we get
	$$
	q_{0}^* (\alpha_0 , \beta_0 ) = q_{\infty}^* (\alpha_{\infty},\beta_{\infty}) = \frac{2N}{N-2}
	$$
	and thus we get existence results for nonlinearities $g(t)= \min \{ t^{q_1 -1} , t^{q_2 -1} \}$ with $ 4 < q_1 < \frac{2N}{N-2}< q_2$.
\end{exa}

\begin{exa} \label{ex2}
	Assume that 
	$$K(r) = \frac{1}{r^{\gamma}} \quad \quad \mbox{for} \,\, r\in (0,1), \quad \quad 
	K(r) = \frac{1}{r^{\delta}} \quad \quad \mbox{for} \,\, r>1, $$
	where $0< \delta < \gamma <2$, and suppose that $V$ just satisfies $\bf (H)$, without any further assumption on its asymptotic behavior as $r \rightarrow +\infty$.
	
	Computing the coefficients $b, b_0$ in \cite{Uberlandio2}, one gets $b \geq -\delta$ and $b_0=-\gamma$, whence $b >b_0$. As a consequence, the results of \cite{Uberlandio2} cannot be applied, because they need $b_0\geq b$. If we set $\beta_0 = \beta_{\infty}=0$, $\alpha_0 =-\gamma$, $\alpha_{\infty}=- \delta$, we get 
	$$ 
	q_{0}^* (\alpha_0 , \beta_0 ) = 2\, \frac{N-\gamma}{N-2}, \quad q_{\infty}^* (\alpha_{\infty},\beta_{\infty}) =2\, \frac{N-\delta}{N-2} ,
	$$
	so that, by the hypotheses on $\delta$ and $\gamma$,
	$$ 
	2< q_{0}^* (\alpha_0 , \beta_0 ) <q_{\infty}^* (\alpha_{\infty},\beta_{\infty}) .
	$$
	Hence we can apply our existence result to nonlinearities of the form $g(t)= \min \{ t^{q_1 -1} , t^{q_2 -1} \}$ with
	$$4 < q_1 < 4 \,\frac{N-\gamma}{N-2}<4 \,\frac{N-\delta}{N-2}<q_2 .$$
	
\end{exa}

\begin{exa} \label{ex3}
	Assume that there exists $\gamma >1$ such that
	$$K(r) = \frac{1}{r} \quad {\mbox as} \; r \rightarrow 0^+ , \quad \quad  K(r) 
	=\frac{1}{r^{\gamma}} \quad {\mbox as} \; r \rightarrow +\infty $$
	and, as in the previous example, suppose that $V$ just satisfies $\bf (H)$, without further assumptions on its behavior at $+\infty$. 
	
	Computing the coefficients $a_0, b_0$ of \cite{Uberlandio2}, one gets $a_0=-2$ and $b_0=-1$, so that \cite[Theorem 1.1]{Uberlandio2} cannot be applied, because it needs $a_0 \geq b_0$. If we set $\beta_0 = \beta_{\infty}=0$, $\alpha_0 = -1$ and $\alpha_{\infty}=-\gamma$, we get 
	$$ 
	q_{0}^* (\alpha_0 , \beta_0 ) = \,\frac{-2 +2N}{N-2}> 2, \quad q_{\infty}^* (\alpha_{\infty},\beta_{\infty})=2\, \frac{N-\gamma}{N-2} <q_{0}^* (\alpha_0 , \beta_0 ) .
	$$
	\noindent So we can apply our existence result to nonlinearities $g(t)= \min \{ t^{q_1 -1} , t^{q_2 -1} \}$ with 
	$$4 < q_1 < 4 \,\frac{N-1}{N-2}, \quad 4 \, \frac{N-\gamma}{N-2} <q_2 , \quad q_1 \leq q_2 .$$
	In particular, setting $\bar{q} = \max \left\{ 4, 4\frac{N-\gamma}{N-2} \right\}$, we can choose 
	$q_1 =q_2 =q \in \left( \bar{q}, 4 \,\frac{N-1}{N-2} \right)$ and we get an existence result for a standard power-type nonlinearity $g(t)= t^{q -1} $. If $\gamma \geq 2$, we have $\bar{q}=4$ and such a result holds for $q \in \left( 4 , 4 \, \frac{N-1}{N-2} \right)$.
	
\end{exa}

\begin{exa} \label{ex4}
	Assume that
	$$K(r) = \frac{1}{r^{\gamma}} \quad \quad \mbox{for} \,\, r\in (0,1), \quad \quad K(r) = \frac{1}{r^{\delta}} \quad \quad \mbox{for} \,\, r>1, $$
	with $0< \delta  $ and $0< \gamma <2$, and that
	$$
	V(r)= \frac{1}{r^2} \quad \quad \mbox{for} \,\, r\in (0,1), \quad \quad V(r) = e^{-r} \quad \quad \mbox{as} \,\, r\rightarrow +\infty .
	$$
	The results of \cite{Uberlandio2} cannot be applied, because they need a power-type decay of $V(r)$ at infinity. Instead, we can apply our results exactly as in Example \ref{ex2}, setting $\beta_0 =\beta_{\infty} =0$, $\alpha_0 = -\gamma $ and $\alpha_{\infty} = - \delta$. Then we get existence results for nonlinearities $g(t)= \min \{ t^{q_1 -1} , t^{q_2 -1} \}$ with
	$$4 < q_1 < 4 \,\frac{N-\gamma}{N-2}, \quad 4 \,\frac{N-\delta}{N-2}<q_2 , \quad q_1 \leq q_2.$$
	If $\delta >2$, we can pick $q_1 =q_2 = q \in \left( 4, \, 4 \, \frac{N-\gamma}{N-2} \right)$ and such results thus concern nonlinearities $g(t) = t^{q-1}$.

\end{exa}

\begin{exa} \label{ex5}
	Assume that
	$$K(r) = \frac{e^r}{r} \quad \quad \mbox{for all} r>0, $$
	and
	$$
	V(r)= \frac{1}{r^2} \quad \quad \mbox{for} \,\, r\in (0,1), \quad \quad V(r) = e^{r} \quad \quad \mbox{as} \,\, r\rightarrow +\infty .
	$$
	The results of \cite{Uberlandio2} cannot be applied, because they need a power-type growth of $K(r)$ at infinity. We can apply our results setting 
	$\alpha_0 = \alpha_{\infty}=-1 $, $\beta_0 = 0$ and $\beta_{\infty} =1$. Then we get 
	$$ 
	q_{0}^* (\alpha_0 , \beta_0 ) = \,\frac{-2 +2N}{N-2}> 2, \quad q_{\infty}^* (\alpha_{\infty},\beta_{\infty})=2\, \frac{N-3}{N-2},
	$$
	which gives existence results for nonlinearities $g(t)= \min \{ t^{q_1 -1} , t^{q_2 -1} \}$ with 
	$$4< q_1 < 4\, \frac{N-1}{N-2}, \quad q_2 >4\, \frac{N-3}{N-2}, \quad q_1 \leq q_2.$$ 
	In particular we can choose $g(t)= t^{q-1}$ with 
	$q_1 =q_2 =q \in \left( 4, \, 4\, \frac{N-1}{N-2} \right)$. 
	
\end{exa}

\begin{exa} \label{ex6}
	As a last example, assume
	$$K(r) = |\log r | \quad \quad \mbox{for} \,\, r\rightarrow 0^+, \quad \quad K(r) = \frac{1}{r^{2}} \quad \quad \mbox{for} \,\, r>1 $$
	and suppose that $V$ satisfies nothing other than $\bf (H)$. The results of \cite{Uberlandio2} cannot be applied, because they need a power-type behavior of $K(r)$ at zero. On the one hand, we can assume  $\beta_{\infty} =0$ and  $\alpha_{\infty} =-2$, which gives 
	$q_{\infty}^* (\alpha_{\infty},\beta_{\infty})=2.$
	On the other hand, fix $q \in \left( 4, \, 4 \, \frac{N}{N-2} \right)$ and pick $\alpha_0 <0$ small enough that 
	$q <\, 4 \, \frac{N+ \alpha_0}{N-2}$. Together with $\beta_0 =0$, this gives
	$$q_{0}^* (\alpha_{0},\beta_{0})=\, 2\, \frac{N+ \alpha_0}{N-2}.$$
	So we obtain existence results for every $4 < q_1 < 4\, \frac{N+ \alpha_0}{N-2}$ and $q_2 >4$. As we can take $q_1 = q_2 =q$ (where $q$ is the same we have choosen before), we get an existence result for every nonlinearity $g(t)= t^{q-1}$ with $q \in \left( 4, \, 4 \, \frac{N}{N-2} \right)$.
	
\end{exa}


\begin{thebibliography}{99}


 \bibitem{AiresSouto}
J.F. Aires and M.A. Souto,
{\it Existence of solutions for a quasilinear Schr\"odinger 
equation with vanishing potentials,}
 J. Math. Anal. Appl., 416  (2014), 924-946.
 
 \bibitem{AubinEkeland}
J.P. Aubin and I. Ekeland, 
``Applied Nonlinear Analysis,"
 Courier Corporation 2006.
 
 \bibitem{BGR_quasi}
 M. Badiale, M. Guida, and S. Rolando, 
 {\it Existence results for a class of quasilinear Schr\"odinger 
  equations with singular or vanishing potentials,}
Nonlinear Anal., 220 (2022), 112816.

\bibitem{BPR}
M. Badiale, L. Pisani,  and S. Rolando, 
{\it Sum of weighted Lebesgue spaces and nonlinear elliptic 
equations,}
 NoDEA, Nonlinear Differ. Equ. Appl., 18 (2011), 369-405.
 
 
 \bibitem{Brandi-et}
 H. Brandi, C. Manus, G. Mainfray, T. Lehner, and G. Bonnaud, 
 \textit{Relativistic and ponderomotive self-focusing of a laser beam in a radially inhomogeneous plasma, }
 Phys. Fluids, B5 (1993), 3539–3550.
 
 
 \bibitem{ColinJeanjean}
 M. Colin and L. Jeanjean,
 \textit{Solutions for a quasilinear Schr\"{o}dinger equation: a dual approach},
 Nonlinear Anal., 56 (2004), 213-226.
 
  
 \bibitem{Uberlandio2}
 G.M. de Carvalho and U.B. Severo,
\textit{ Quasilinear Schr\"{o}dinger equations with unbounded or decaying potentials},
 Math. Nachr., 291 (2018), 492-517.
  
 
 \bibitem{FangSzulkin}
 X.D, Fang and A. Szulkin,
 \textit{Multiple solutions for a quasilinear Schr\"{o}dinger equation},
 J. Differential Equations, 254 (2013), 2015-2032.
  
 
 \bibitem{FurtadoSilvaSilva}
 M.F. Furtado, E.D. da Silva, and M.L. Silva,
\textit{ Quasilinear elliptic problems under asymptotically linear conditions at infinity and at the origin},
 Z. Angew. Math. Phys., 66 (2015), 277-291.
 
 
 \bibitem{Gloss}
 E. Gloss,
 \textit{Existence and concentration of positive solutions for a quasilinear equation in $R^N$},
 J. Math. Anal. Appl., 371 (2010), 465-484.
 
 
 \bibitem{Kuri}
 S. Kurihara,
 \textit{Large-amplitude quasi-solitons in superfluids films},
 J. Phys. Soc. Japan, 50 (1981), 3262–3267.
 
 \bibitem{Kwon}
 O. Kwon, 
 \textit{Nonexistence of positive solutions for quasilinear equations with decaying potentials},
 Mathematics, 8 (2020), 425. 
 
 \bibitem{Li-Huang}
 G. Li and Y. Huang,
 \textit{Positive solutions for critical quasilinear Schrödinger equations with potentials vanishing at infinity},
 Discrete Contin. Dyn. Syst. B, 27 (2022), 3971-3989.  
 
 \bibitem{LiuWang}
 J. Liu and Z.Q. Wang,
 \textit{Soliton solutions for quasilinear Schr\"{o}dinger equations. I},
 Proc. Amer. Math. Soc., 131 (2003), 441-448.
  
 \bibitem{LiuLiuWang1}
 X.Q. Liu, J.Q. Liu, and Z.Q. Wang,
\textit{ Quasilinear elliptic equations via perturbation method},
 Proc. Amer. Math. Soc., 141 (2013), 253-263.
  
 \bibitem{LiuLiuWang2}
 X.Q. Liu, J.Q. Liu, and Z.Q. Wang, 
\textit{ Quasilinear elliptic equations with critical growth via perturbation method},
 J. Differential Equations, 254 (2013), 102-124.
 
 \bibitem{PoppenbergSchmittWang}
 M. Poppenberg, K. Schmitt, and Z.Q. Wang,
 \textit{On the existence of soliton solutions to quasilinear Schr\"{o}dinger equations}, 
 Calc. Var. Partial Differential Equations, 14 (2002), 329-344.
 
 \bibitem{SilvaVieira}
 E.B. Silva and  G.F. Vieira,
\textit{ Quasilinear asymptotically periodic Schr\"{o}dinger equations with critical growth},
 Calc. Var. Partial Differential Equations, 39 (2010), 1-33.
 
  \bibitem{YangWangZhao}
 X. Yang, W. Wang, and F. Zhao,
 \textit{Infinitely many radial and non-radial solutions to a quasilinear Schr\"{o}dinger equation},
 Nonlinear Anal., 114 (2015), 158-168.
 
  \bibitem{Yang-Zhang}
 Y. Yang and J. Zhang, 
 \textit{A note on the existence of solutions for a class of quasilinear elliptic equations: an Orlicz-Sobolev space setting}, 
 Bound. Value Probl., 2012, 2012:136, 7 pages.
  
 \bibitem{Zhang13}  
 G, Zhang,
 \textit{Weighted Sobolev spaces and ground state solutions for quasilinear elliptic problems with unbounded and decaying potentials},
 Bound. Value Probl., 2013, 2013:189, 15 pages.
 

\end{thebibliography}
\end{document}